\documentclass[letterpaper, 10 pt, conference]{ieeeconf}  

\IEEEoverridecommandlockouts                              

\usepackage{cite}
\usepackage{amsmath,amssymb,amsfonts,mathtools}

\usepackage{hyperref}       
\usepackage{xcolor}
\usepackage{bm}
\usepackage{subcaption}

\providecommand{\norm}[1]{\ensuremath{\left\lVert#1\right\rVert }}
\providecommand{\mnorm}[1]{\ensuremath{\left\lvert#1\right\rvert}}

\def\R{\mathbb{R}}
\def\G{\mathbb{G}}
\def\E{\mathcal{E}}
\def\N{\mathcal{N}}
\def\L{\tilde{\mathcal{L}}}
\def\m{i \in \{1,2,\ldots,m\}}

\usepackage[export]{adjustbox}
\usepackage{amsthm}
\newtheorem{theorem}{Theorem}
\newtheorem{lemma}{Lemma}

\newtheorem{remark}{Remark}

\newtheorem{assumption}{Assumption}
\newtheorem{definition}{Definition}

\DeclareRobustCommand{\bigO}{%
  \text{\usefont{OMS}{cmsy}{m}{n}O}%
}

\newcommand{\bx}{{\mathbf x}}
\newcommand{\bX}{{\mathbf X}}
\newcommand{\bv}{{\mathbf v}}
\newcommand{\bz}{{\mathbf z}}
\newcommand{\bZ}{{\mathbf Z}}
\newcommand{\bxs}{{\mathbf{x_*}}}
\newcommand{\bXs}{{\mathbf{X_*}}}
\newcommand{\by}{{\mathbf y}}

\renewcommand{\baselinestretch}{1.0}
\makeatletter
\hypersetup{colorlinks=true}
\AtBeginDocument{\@ifpackageloaded{hyperref}
  {\def\@linkcolor{blue}
  \def\@anchorcolor{red}
  \def\@citecolor{red}
  \def\@filecolor{red}
  \def\@urlcolor{red}
  \def\@menucolor{red}
  \def\@pagecolor{red}
\begingroup
  \@makeother\`%
  \@makeother\=%
  \edef\x{%
    \edef\noexpand\x{%
      \endgroup
      \noexpand\toks@{%
        \catcode 96=\noexpand\the\catcode`\noexpand\`\relax
        \catcode 61=\noexpand\the\catcode`\noexpand\=\relax
      }%
    }%
    \noexpand\x
  }%
\x
\@makeother\`
\@makeother\=
}{}}
\makeatother

\makeatletter
\newcommand\footnoteref[1]{\protected@xdef\@thefnmark{\ref{#1}}\@footnotemark}
\makeatother

\title{\LARGE \bf On Linear Convergence of PI Consensus Algorithm \\ under the Restricted Secant Inequality}

\author{Kushal Chakrabarti \and Mayank Baranwal
\thanks{K.~Chakrabarti  and M.~Baranwal are with the Division of Data \& Decision Sciences, Tata Consultancy Services Research, Mumbai, 400607 India. e-mails: \texttt{\{chakrabarti.k, baranwal.mayank\}@tcs.com}. M.~Baranwal is also with the Faculty of Systems \& Control Engineering, Indian Institute of Technology, Bombay.}}

\begin{document}

\maketitle
\thispagestyle{empty}
\pagestyle{empty}


\begin{abstract}
This paper considers solving distributed optimization problems in peer-to-peer multi-agent networks. The network is synchronous and connected. By using the proportional-integral (PI) control strategy, various algorithms with fixed stepsize have been developed. Two notable among them are the PI algorithm and the PI consensus algorithm. Although the PI algorithm has provable linear or exponential convergence without the standard requirement of (strong) convexity, a similar guarantee for the PI consensus algorithm is unavailable. In this paper, using Lyapunov theory, we guarantee exponential convergence of the PI consensus algorithm for global cost functions that satisfy the restricted secant inequality, with rate-matching discretization, without requiring convexity. To accelerate the PI consensus algorithm, we incorporate local pre-conditioning in the form of constant positive definite matrices and numerically validate its efficiency compared to the prominent distributed convex optimization algorithms. Unlike classical pre-conditioning, where only the gradients are multiplied by a pre-conditioner, the proposed pre-conditioning modifies both the gradients and the consensus terms, thereby controlling the effect of the communication graph on the algorithm. 
\end{abstract}
\section{Introduction}
\label{sec:intro}

We consider solving the multi-agent distributed convex optimization problem over a peer-to-peer network of $m$ agents. Each agent $\m$ in the network can communicate with a certain set of other agents called its {\em neighbors}, denoted by $\N_i$. The inter-agent communication topology is represented by an undirected {\em graph} $\G=(\{1, \ldots, \, m\}, \E)$, with an {\em edge} $(i,j)\in \E$ or $(j,i)\in \E$ if agent $i$ and agent $j$ are {\em neighbors}, for any $i,j \in \{1,\ldots,m\}$, $i \neq j$. Each agent has a {\em local and private cost function} $f_i: \R^d \to \R$. The agents aim to compute a common vector $\bxs \in \R^d$ that minimizes the aggregate cost function held by all the agents:
\begin{align}
    \bxs \in  \arg \min_{\bx \in \R^d} \sum_{i=1}^m f_i(\bx). \label{eqn:opt_1}
\end{align}
We let $f: \R^d \to \R$ and $F: \R^{md} \to \R$, respectively, denote the {\em aggregate cost} and the {\em cumulative cost}, i.e., $f(\bx) = \sum_{i=1}^m f_i(\bx)$ and $F(\bx) = \sum_{i=1}^m f_i(\bx_i)$ for $\bx \in \R^{md}$. 

A class of distributed gradient-based continuous-time algorithms that solve~\eqref{eqn:opt_1} is based on proportional-integral (PI) control law~\cite{yang2019survey}. Here, an integral error term~\cite{kia2015distributed, yi2020exponential} or a consensus term of integral errors~\cite{wang2010control} is fed back to the state dynamics to guarantee consensus and convergence~\cite{hatanaka2018passivity}. We focus on PI consensus algorithm, because we noticed that the PI consensus algorithm has convergence speed comparable or favorable to competing algorithms, e.g., EXTRA, DIGing, PI, that also have provable exponential rate. However, PI consensus can be slower than the more recent double-loop Accelerated-EXTRA~\cite{li2020revisiting} algorithm (see Figure~\ref{fig:mnist}). We find that the PI consensus algorithm can be accelerated by a suitable preconditioning strategy. Details of the preconditioning matrix is presented in Section~\ref{sec:exp}.

Existing works~\cite{kia2015distributed, yi2020exponential, wang2010control, hatanaka2018passivity} used Lyapunov theory to prove asymptotic stability of PI-based algorithms for convex local costs. However, achieving accelerated convergence to a solution of~\eqref{eqn:opt_1} requires a stronger form than just asymptotic stability. Exponential stability of the continuous-time PI algorithm has been proved in~\cite{kia2015distributed} for directed graphs when the local
costs are {\em strongly convex} and in~\cite{yi2020exponential} when the cumulative cost $F$ satisfy the {\em restricted secant inequality} (RSI) condition. RSI is one of the strong-convexity relaxation classes for guaranteeing linear convergence of centralized algorithms~\cite{necoara2019linear}. Quasi-strongly convex or composition of strongly convex function with linear map plus a linear
term satisfy RSI. Despite utilizing PI control, the PI algorithm~\cite{kia2015distributed, yi2020exponential} and the PI consensus~\cite{wang2010control, hatanaka2018passivity} algorithm are different in their estimate updates and the initialization of the integral error states~\cite{yang2019survey}. To highlight this, we briefly present their update strategies below. Each agent $i$ in PI consensus updates its state $\bx_i$ and the integral consensus $\bv_i$ as~\cite{wang2010control}
\begin{align}
    \Dot{\bx}_i \hspace{-0.2em}  = &  \sum_{j \in \N_i}  \hspace{-0.2em} (\bx_j  - \bx_i )  \hspace{-0.2em} -  \hspace{-0.2em} \beta\!\!\sum_{j \in \N_i}  \hspace{-0.2em} (\bv_j  - \bv_i )  \hspace{-0.2em} -  \hspace{-0.2em} \alpha \nabla f_i(\bx_i ), \label{eqn:pic_x} \\
    \Dot{\bv}_i  = & \beta \sum_{j \in \N_i} (\bx_j  - \bx_i ), \label{eqn:pic_v}
\end{align}
where $\alpha, \beta > 0$ are algorithm parameters.
Whereas the PI algorithm is described by~\cite{kia2015distributed}
\begin{align}
    \Dot{\bx}_i \hspace{-0.2em}  = &  \sum_{j \in \N_i}  \hspace{-0.2em} (\bx_j  - \bx_i )  \hspace{-0.2em} -  \hspace{-0.2em} \beta  \bv_i \hspace{-0.2em} -  \hspace{-0.2em} \alpha \nabla f_i(\bx_i ), \label{eqn:pi_x} \\
    \Dot{\bv}_i  = & - \beta \sum_{j \in \N_i} (\bx_j  - \bx_i ). \label{eqn:pi_v}
\end{align}
Specifically, the PI consensus algorithm~\eqref{eqn:pic_x}-\eqref{eqn:pic_v} can be interpreted as a primal-dual algorithm for solving 
\begin{align*}
    \min_{\{\bx_i \in \R^d, \forall i\}} & \sum_{i = 1}^m f_i(\bx_i), \,
    s.t. \,  (L \otimes I) \bx = \mathbf{0_{md}},
\end{align*}
and the PI algorithm~\eqref{eqn:pi_x}-\eqref{eqn:pi_v} is a primal-dual algorithm for 
\begin{align*}
    \min_{\{\bx_i \in \R^d, \forall i\}} & \sum_{i = 1}^m f_i(\bx_i), \,
    s.t. \, (L \otimes I)^{1/2} \bx = \mathbf{0_{md}},
\end{align*}
where $I$ denotes the $(d \times d)$-dimensional identity matrix,
$L$ denotes the Laplacian matrix of the graph $\G$,
$\otimes$ denotes the Kronecker product, and $\bx  = [\bx_1 ^{\top} \ldots \bx_m ^{\top}]^{\top}$.
So, each agent in PI consensus~\eqref{eqn:pic_x}-\eqref{eqn:pic_v} shares both $\bx_i$ and $\bv_i$ with its neighbors, which allows the agents to incorporate the integral consensus term $\sum_{j \in \N_i} (\bv_j  - \bv_i )$ in its $\bx_i$-update. On the other hand, each agent in PI~\eqref{eqn:pi_x}-\eqref{eqn:pi_v} shares only $\bx_i$ with its neighbors, and does not use $\bv_j$ to update its $\bx_i$. Moreover, the convergence of the PI algorithm~\eqref{eqn:pi_x}-\eqref{eqn:pi_v} requires each $\bv_i(0)$ to be initialized as the zero vector, whereas the PI consensus~\eqref{eqn:pic_x}-\eqref{eqn:pic_v} allows arbitrary initialization of each $\bv_i(0)$. So, PI consensus~\eqref{eqn:pic_x}-\eqref{eqn:pic_v} is robust to the initialization of $\bm{\bv_i(0)}$~\cite{yi2020exponential}, with both algorithms having communication complexity $\bm{\bigO(d)}$. Exponential stability of the continuous-time PI consensus~\eqref{eqn:pic_x}-\eqref{eqn:pic_v} has been proved in~\cite{liang2019exponential} for metrically subregular primal-dual gradient maps and convex local costs and in~\cite{sun2020control} for strong convex $f$. 

Hence, while exponential stability of the PI algorithm~\eqref{eqn:pi_x}-\eqref{eqn:pi_v} has been proved in~\cite{yi2020exponential} when $F$ satisfies the RSI condition, exponential stability of the PI consensus algorithm~\eqref{eqn:pic_x}-\eqref{eqn:pic_v} for the same class of cost functions has not been proved in the literature. We prove exponential stability of the continuous-time PI consensus algorithm for undirected graphs when $F$ satisfies the RSI condition with Lipschitz continuous local gradients without requiring convexity of the cost function. Note that an even stronger notion of convergence is fixed-time convergence~\cite{garg2020fixed} in continuous-time.

Despite key insights from studying optimization algorithms as continuous-time dynamics, there is substantial literature on discrete-time distributed algorithms. In the seminal distributed gradient-descent (DGD) algorithm,
each agent combines its local gradient and consensus terms to update the local estimate of $\bxs$~\cite{nedic2009distributed}. Notable discrete-time algorithms that are built upon DGD include ADMM~\cite{shi2014linear}, EXTRA~\cite{shi2015extra}, DIGing~\cite{nedic2017achieving}, PI~\cite{kia2015distributed}, APM-C~\cite{li2018sharp}, Mudag~\cite{ye2020multi}, Accelerated EXTRA~\cite{li2020revisiting}, DAccGD~\cite{rogozin2021towards}, and ACC-SONATA~\cite{tian2022acceleration}. For convex local costs, asymptotic convergence of DGD, ADMM, EXTRA, DIGing, PI, and Accelerated EXTRA is proved, assuming Lipschitz continuous or bounded local gradients. Linear convergence of these algorithms requires stronger assumptions, such as restricted strong convexity of the aggregate cost for EXTRA~\cite{shi2015extra}, {\em restricted secant inequality} for the PI algorithm~\cite{yi2020exponential}, strong convexity of each local cost~\cite{shi2014linear, nedic2017achieving, qu2017harnessing}. The implementation itself of APM-C, Mudag, DAccGD, and ACC-SONATA requires the value of the strong-convexity coefficient of the cumulative or the aggregate cost function. For the first time, we prove linear convergence of the discrete-time PI consensus algorithm for undirected graphs when $F$ satisfies the RSI condition with Lipschitz continuous local gradients without requiring convexity of individual local costs. Thus, compared to the analysis of PI algorithm~\eqref{eqn:pi_x}-\eqref{eqn:pi_v} in~\cite{yi2020exponential}, we have a similar contribution for the PI consensus algorithm~\eqref{eqn:pic_x}-\eqref{eqn:pic_v}.

The key contributions of our work are summarized below.

\begin{itemize}
    \item Although the PI algorithm~\eqref{eqn:pi_x}-\eqref{eqn:pi_v} has provable linear or exponential convergence when the cost function satisfies the {\em restricted secant inequality}, without the standard requirement of (strong) convexity, the distributed optimization literature lacks a similar guarantee for the PI consensus algorithm~\eqref{eqn:pic_x}-\eqref{eqn:pic_v} when solving~\eqref{eqn:opt_1} with the same class of cost functions. We aim to address this problem. For the first time, we rigorously prove exponential convergence of the continuous-time and discrete-time PI consensus algorithm when $F$ satisfies the {\em restricted secant inequality} condition with Lipschitz continuous local gradients without requiring convexity. Our analyses also apply to the case of local pre-conditioning with constant positive-definite matrices. The details are in Section~\ref{sec:lyap}.

    \item In Section~\ref{sec:exp}, we propose a choice of pre-conditioning and numerically show its efficacy. Pre-conditioning potentially increases the effective graph connectivity by modifying the edge weights, leading to faster consensus.
    
    
\end{itemize}

Due to their difference in local estimate update, as described in~\eqref{eqn:pic_x}-\eqref{eqn:pi_v}, analysis of the PI algorithm~\eqref{eqn:pi_x}-\eqref{eqn:pi_v} in~\cite{yi2020exponential} does not trivially apply to PI consensus~\eqref{eqn:pic_x}-\eqref{eqn:pic_v}. Thus, the novelty of our paper lies in proving exponential or linear convergence of PI consensus~\eqref{eqn:pic_x}-\eqref{eqn:pic_v}, under relaxed assumption than the existing analyses of~\eqref{eqn:pic_x}-\eqref{eqn:pic_v}. Furthermore, we propose a local pre-conditioning that can potentially reduce the convergence time of the existing PI consensus method.
\section{Assumptions and Preliminaries}
\label{sec:prelim}


{\bf Notation}:
Let $N$ be any natural number. We use $\nabla g$ to denote the gradient of a function $g: \R^N \to \R$. We let $\norm{\bv}$ denote the Euclidean norm of $\bv \in \R^N$ and $\norm{M}$ denote the induced $2$-norm of a matrix $M$. 
We let $\mathbf{0_N}$ denote the $N$-dimensional zero vector.
We denote a block-diagonal matrix of appropriate dimensions by $Diag(.)$. 
We use the abbreviation SPD for symmetric positive definite. 
We let $\overline{\lambda}_L$ and $\underline{\lambda}_L$ denote the largest and the smallest non-zero eigenvalue of the Laplacian $L$. We let $\bXs = [(\bxs)^{\top}, \ldots, (\bxs)^{\top}]^{\top}$. Finally, we let $\L = L \otimes I$.

\begin{definition}
A differentiable function $g: \R^N \to \R$ satisfies the restricted secant inequality (RSI) with respect to $\bxs$ with constant $\mu > 0$ if $(\nabla g(\bx) - \nabla g(\bxs))^{\top} (\bx-\bxs) \geq \mu \norm{\bx-\bxs}^2 \, \forall \bx \in \R^N$, where $\bxs$ is the unique global minimizer of $g$~\cite{yi2020exponential, shi2015extra}.
\end{definition}
\begin{definition}
A sequence $\{\bx(k)\}\!\subset\!\R^N$ converges linearly to $\bxs\!\in\!\R^N$ with rate $\rho\!\in\!(0,1)$ if there exist constants $C, K > 0$ such that $\norm{\bx(k)\!-\!\bxs} \leq C \rho^{k-K}, \forall k \geq K$.
\end{definition}

\begin{assumption} \label{assump_1}
$\mnorm{\min_{\bx \in \R^d} {f(\bx)}} < \infty$ and the solution set of problem~\eqref{eqn:opt_1} is non-empty.
\end{assumption}\vspace{-.5em}

\begin{assumption} \label{assump_2}
Each $f_i$ is continuously differentiable. 
\end{assumption}\vspace{-.5em}

\begin{assumption} \label{assump_3}
$\G$ is undirected and connected.
\end{assumption}\vspace{-.5em}

\begin{assumption} \label{assump_4}
$F$ satisfies the restricted secant inequality with respect to $\bxs$ with a constant $\mu > 0$, and its gradient $\nabla F$ is $L_f$-Lipcshitz continuous ($F$ is smooth).
\end{assumption}

Assumption~\ref{assump_4} is weaker than the existing works on the analysis of PI consensus algorithm~\eqref{eqn:pic_x}-\eqref{eqn:pic_v}, such as~\cite{gharesifard2013distributed, hatanaka2018passivity}, in the sense that it does not require convexity of the cost function. It further implies that every stationary point is a global minimizer, i.e., the solution set is $\{\bxs \in \R^d | \sum_{i=1}^m \nabla f_i(\bxs) = 0_d\}$.


The following result is standard in Lyapunov stability theory of autonomous systems~\cite{khalil1996nonlinear}.


\begin{lemma} \label{lem:exp}
Consider the system $\Dot{\bx} = \mathbf{g}(\bx)$ where $\bx \in \R^N, \mathbf{g}: \R^N \to \R^N, \mathbf{g}(\mathbf{0_N})=\mathbf{0_N}$, and $\mathbf{g}$ is Lipschitz continuous over $\R^N$.
Let $V: \R^N \to \R$ be a continuously differentiable positive definite function such that $k_1 \norm{\bx}^a \leq V(\bx) \leq k_2 \norm{\bx}^a$ and $\Dot{V}(\bx) \leq - k_3 \norm{\bx}^a$ for any $\bx \in \R^N$, where $k_1, k_2, k_3, a$ are positive constants. Then, the origin is globally exponentially stable, i.e., $\norm{\bx(t)} \leq (\frac{k_2}{k_1})^{1/a} e^{-\frac{k_3}{k_2 a} t} \norm{\bx(0)}, \, \forall \bx(0) \in \R^N$.
\end{lemma}
\section{Convergence of PI Consensus Algorithm}
\label{sec:lyap}

For simplicity in presenting our results, we assume that the solution of~\eqref{eqn:opt_1}, defined in Remark~1, is unique, denoted by $\bxs$. Later, we discuss the applicability of our results in the case when the solution is not unique.

The algorithm presented next is built upon the distributed PI consensus algorithm~\cite{wang2010control, hatanaka2018passivity} for solving~\eqref{eqn:opt_1}. However, the difference is that each agent $\m$ in this algorithm multiplies the local gradient and the consensus terms with a fixed pre-conditioner matrix $K_i$. Each $K_i$ is SPD and chosen by the agents independently before the algorithm begins. Later, in the next section, we propose a choice of pre-conditioning and numerically show its efficacy in improving the convergence time.


At each $t \geq 0$, each agent $\m$ in the (pre-conditioned) PI consensus algorithm updates its state $\bx_i $ and the integral consensus $\bv_i $ according to
\begin{align}
    \Dot{\bx}_i \hspace{-0.2em}  = &  K_i \Bigg(\!\sum_{j \in \N_i}  \hspace{-0.2em} (\bx_j  - \bx_i )  \hspace{-0.2em} -  \hspace{-0.2em} \beta\!\!\sum_{j \in \N_i}  \hspace{-0.2em} (\bv_j  - \bv_i )  \hspace{-0.2em} -  \hspace{-0.2em} \alpha \nabla f_i(\bx_i )\!\Bigg), \label{eqn:ctpi_x} \\
    \Dot{\bv}_i  = & \beta K_i \sum_{j \in \N_i} (\bx_j  - \bx_i ). \label{eqn:ctpi_v}
\end{align}
Here, $\alpha$ and $\beta$ are positive scalar parameters of the algorithm, whose values are presented later in this section. For convenience, we combine the agents' dynamics in a matrix form as follows. We let $\bx  = [\bx_1 ^{\top} \ldots \bx_m ^{\top}]^{\top}$ and $\bv  = [\bv_1 ^{\top} \ldots \bv_m ^{\top}]^{\top}$ respectively denote the agents' combined state vector and the integral consensus at each time $t \geq 0$. We define the combined pre-conditioner $K = Diag\left(\left\{K^i\right\}_{i=1}^m\right)$. Note that $K$ is SPD, as each $K_i$ is SPD. Then,~\eqref{eqn:ctpi_x}-\eqref{eqn:ctpi_v} can be rewritten:
\begin{align}
    \begin{bmatrix} \Dot{\bx}  \\ \Dot{\bv}  \end{bmatrix} & = - \begin{bmatrix} K \L \bx  - \beta K \L \bv  + \alpha K \nabla F(\bx ) \\ \beta K \L \bx  \end{bmatrix}. \label{eqn:ctpi_comb}
\end{align}
Recall the definition of $\bXs$ from Section~\ref{sec:prelim}.
Since $K$ is SPD, it can be verified that, under Assumptions~\ref{assump_1}-\ref{assump_4}, if $(\bXs,\mathbf{v}_*)$ is an equilibrium point of~\eqref{eqn:ctpi_comb}, then $\bxs$ solves~\eqref{eqn:opt_1}.




Upon utilizing LaSalle's invariance principle~\cite{khalil1996nonlinear}, under Assumptions~\ref{assump_1}-\ref{assump_4}, it can be proved that each agent's estimate in Algorithm~\eqref{eqn:ctpi_x}-\eqref{eqn:ctpi_v} with $\bx(0), \bv(0) \in \R^{md}$ asymptotically converges to the same solution $\bxs$ of~\eqref{eqn:opt_1}. It does not require the convexity of individual local costs. The proof for asymptotic stability uses the Lyapunov function $V: \R^{2md} \to \R$ such that $V(\bz,\by) = \frac{1}{2}\left(\bz^{\top} K^{-1} \bz + \by^{\top} K^{-1} \by\right)$ for $\bz,\by \in \R^{md}$ and follows the proof of Theorem~1 in~\cite{hatanaka2018passivity}. So, LaSalle's principle guarantees asymptotic stability of~\eqref{eqn:ctpi_comb}, but with an unknown rate. The following analysis implies that Lyapunov theory (Lemma~\ref{lem:exp}) can guarantee stronger stability for the class of $F$ that satisfies Assumption~\ref{assump_4}.

\begin{theorem} \label{thm:thm2}
Consider algorithm~\eqref{eqn:ctpi_comb} with initial condition $\bx(0), \bv(0) \in \R^{md}$. Suppose that Assumptions~\ref{assump_1}-\ref{assump_4} hold true. Then there exists $\overline{\alpha}, \overline{\beta}, \underline{\beta} \in (0,\infty)$ such that, for $\alpha < \overline{\alpha}$ and $\beta \in (\overline{\beta},\underline{\beta})$, the local estimate $\bx_i $ exponentially converges to the same $\bxs$ for each agent $\m$.
\end{theorem}

\begin{proof}
We define the estimation errors at time $t \geq 0$ as
$\bz  = \bx  - \bXs, \, \by  = \bv  - \mathbf{v_*}$.
Then, from~\eqref{eqn:ctpi_comb},
\begin{align}
    \begin{bmatrix} \Dot{\bz} \\ \Dot{\by}  \end{bmatrix} \hspace{-0.3em} = \hspace{-0.3em}
    - \begin{bmatrix} K (\L \bz  - \beta \L \by  + \alpha \nabla F(\bx ) - \alpha \nabla F(\bXs)) \\ \beta K \L \bz  \end{bmatrix}. \label{eqn:ctpi_err}
\end{align}
Consider the case $\L \by  \neq \mathbf{0_{md}}$. We define Lyapunov candidate $V: \R^{2md} \to \R$ such that 
$$V(\bz,\by) = \frac{c_1}{2} \bz^{\top} K^{-1} \bz + \frac{c_2}{2} \by^{\top} K^{-1} \by - c_3 \bz^{\top} K^{-1} \by,$$
where $c_1, c_2, c_3 > 0$. Along the trajectories of~\eqref{eqn:ctpi_err},
$\Dot{V} =  c_1 \bz ^{\top} K^{-1} \Dot{\bz}  + c_2 \by ^{\top} K^{-1} \Dot{\by} - c_3 \bz ^{\top} K^{-1} \Dot{\by}  - c_3 \by ^{\top} K^{-1} \Dot{\bz}$.
Upon substituting from~\eqref{eqn:ctpi_err}, $\Dot{V} = - (c_1 - \beta c_3) \bz ^{\top} \L \bz  - \alpha c_1 \bz ^{\top} (\nabla F(\bx ) - \nabla F(\bXs))  + (\beta c_1 - \beta c_2 + c_3) \by ^{\top} \L \bz - c_3 (\beta \by ^{\top} \L \by  - \alpha \by ^{\top} (\nabla F(\bx ) - \nabla F(\bXs)))$.
Suppose, $(c_1 - \beta c_3) > 0$. We use the following result from Lemma~2 of~\cite{yi2020exponential}. Under Assumptions~\ref{assump_1}-\ref{assump_4}, there exist $\overline{\alpha} > 0, \mu_1 = \min\{ \frac{\mu}{2m} \alpha c_1, (c_1- \beta c_3) \underline{\lambda}_L - \frac{2m L_f^2 + \mu \alpha c_1 L_f}{\mu \alpha c_1}\}$ s.t. if $\alpha < \overline{\alpha}$,
\begin{align}
    & \alpha c_1 (\nabla F(\bx) - \nabla F(\bXs))^{\top} (\bx-\bXs) + (c_1 - \beta c_3) \bx^{\top} \L \bx \nonumber \\
    & \geq \mu_1 \norm{\bx - \bXs}^2, \, \forall \bx \in \R^{md}. \label{eqn:lem2}
\end{align}
Following the proof of Theorem 1 in~\cite{yi2020exponential}, from above, $\Dot{V} \leq - \mu_1 \norm{\bz }^2 + (\beta c_1 - \beta c_2 + c_3) \by ^{\top} \L \bz  - c_3 (\beta \by ^{\top} \L \by  - \alpha \by ^{\top} (\nabla F(\bx ) - \nabla F(\bXs)))$.
If $(\beta c_1 - \beta c_2 + c_3) = 0$,
\begin{align*}
    & \Dot{V} \leq - \mu_1 \norm{\bz }^2 - \beta c_3 \underline{\lambda}_L \norm{\by }^2 + \alpha c_3 \by ^{\top} (\nabla F(\bx ) - \nabla F(\bXs)) \\
    & \leq - \mu_1 \norm{\bz }^2 - \beta c_3 \underline{\lambda}_L \norm{\by }^2 + \alpha c_3 \norm{\by} \norm{\nabla F(\bx ) - \nabla F(\bXs)} \\
    & \leq - \mu_1 \norm{\bz }^2 - \beta c_3 \underline{\lambda}_L \norm{\by }^2 + \alpha c_3 p \norm{\by }^2 \\
    & + \frac{\alpha c_3}{p} \norm{\nabla F(\bx ) - \nabla F(\bXs)}^2 \\
    & \leq - (\mu_1 - \frac{\alpha c_3 L_f^2}{p}) \norm{\bz }^2 - c_3 (\beta \underline{\lambda}_L - \alpha p) \norm{\by }^2,
\end{align*}
where $p > 0$ and the third inequality follows from $\nabla F$ being $L_f$-Lipschitz continuous under Assumption~\ref{assump_4}. 
Suppose that $(\mu_1 - \frac{\alpha c_3 L_f^2}{p}) > 0$ and $(\beta \underline{\lambda}_L - \alpha p) > 0$. From the definition of $V$, $c_1 c_2 > c_3^2$ implies that $V(\bz,\by) > 0$ for any $\begin{bmatrix} \bz^{\top} & \by^{\top} \end{bmatrix}^{\top} \neq \mathbf{0_{2md}}$ and $V(\mathbf{0_{2md}}) = 0$. Hence, if  the four positive scalar parameters $c_1, c_2, c_3, \beta$ satisfy
\begin{align}
    c_1 c_2 > c_3^2, c_1 > \beta c_3, \alpha^2 L_f^2 c_3 < \beta \mu_1 \underline{\lambda}_L, \beta c_1 - \beta c_2 + c_3 = 0, \label{eqn:ineq}
\end{align}
then one can choose $p \in (\frac{\alpha L_f^2 c_3}{\mu_1}, \frac{\beta \underline{\lambda}_L}{\alpha})$, so that 
\begin{align}
    \Dot{V} \hspace{-0.3em} \leq \hspace{-0.3em} & - \hspace{-0.3em} \min\{\mu_1 \hspace{-0.3em} - \hspace{-0.3em} \frac{\alpha c_3 L_f^2}{p}, \hspace{-0.2em} c_3 (\beta \underline{\lambda}_L - \alpha p) \hspace{-0.2em}\} \hspace{-0.2em} (\norm{\bz }^2 + \norm{\by }^2), \hspace{-0.3em} \label{eqn:Vdot_bd} \\
    V \geq & \underline{\epsilon} (\norm{\bz }^2 + \norm{\by }^2), \label{eqn:V_bd_low}
\end{align}
for some $\underline{\epsilon} > 0$. By definition, $V(\bz,\by) \leq \norm{K^{-1}} (\frac{c_1}{2} \norm{\bz}^2 + \frac{c_2}{2} \norm{\by}^2 + \frac{c_3}{2} (\norm{\bz}^2 + \norm{\by}^2))$, implying that
\begin{align}
    V(\bz ,\by ) \leq \overline{\epsilon} (\norm{\bz }^2 + \norm{\by }^2), \label{eqn:V_bd_upp}
\end{align}
for some $\overline{\epsilon} > 0$. So, the conditions of Lemma~\ref{lem:exp} hold. 


Next, we consider the other case $\L \by  = \mathbf{0_{md}}$. The error dynamics~\eqref{eqn:ctpi_err} is reduced to $\Dot{\bz}  = - K \L \bz  - \alpha K (\nabla F(\bx) - \nabla F(\bXs))$. Consider the Lyapunov candidate $V_2: \R^{md} \to \R$ such that $V_2(\bz) = \frac{1}{2} \bz^{\top} K^{-1} \bz$ for $\bz \in \R^{md}$. Along the above trajectories,
$\Dot{V}_2(\bz ) = \bz ^{\top} K^{-1} \Dot{\bz} = - \bz ^{\top} \L \bz  - \alpha \bz ^{\top} (\nabla F(\bx ) - \nabla F(\bXs))$.
Under Assumptions~\ref{assump_1}-\ref{assump_4}, we have similar result as~\eqref{eqn:lem2}: there exist $\overline{\overline{\alpha}}, \mu_2 > 0$ such that if $\alpha < \overline{\overline{\alpha}}$ then,
\begin{align*}
    & \alpha (\nabla F(\bx) - \nabla F(\bXs))^\top (\bx-\bXs) + \bx^\top \L \bx \geq \mu_2 \norm{\bx - \bXs}^2\!\!\!.
\end{align*}
Then, $\Dot{V}_2(\bz ) \leq - \mu_2 \norm{\bz }^2$. Moreover, $V_2(\bz)$ is quadratic in $\bz \in \R^{md}$. So, the conditions of Lemma~\ref{lem:exp} hold.

Since we have a switched dynamics of~\eqref{eqn:ctpi_err} depending on whether $\by(t)$ is in the null-space of $\L$, we make the following argument to conclude exponential stability of~\eqref{eqn:ctpi_err}. 
Consider any time $t=t_1$ when the Lyapunov function $V_2$ is active. Since $V_2(\bz)$ is quadratic for any $\bz$, $V_2(t_1) \geq k_1 \norm{\bz(t_1)}^2$ and $V_2  \leq k_2 \norm{\bz }^2$ for some $k_1, k_2 > 0$ and for all $t$. Moreover, $\bz $ exponentially decreases when either $V$ or $V_2$ is active. So, there exists $t_2 \in [t_1, \infty)$ such that $V_2(t_2) \leq k_2 \norm{\bz(t_2)}^2 \leq k_1 \norm{\bz(t_1)}^2 \leq V_2(t_1)$. Hence, there exists a subsequence of $\R_{>0}$ over which $V_2$ is exponentially decreasing. Owing to Lemma~\ref{lem:exp}, the proof is complete.
\end{proof}

Next, we verify the feasibility of~\eqref{eqn:ineq}.  We have
\begin{align*}
    & \beta c_1 - \beta c_2 + c_3 = 0, c_1 c_2 > c_3^2 \iff \beta c_1^2 + c_1 c_3 > \beta c_3^2, \\
    & c_1 > \beta c_3 \iff \beta c_1^2 + c_1 c_3 > (\beta^2 + 1) c_1 c_3.   
\end{align*}
$(\beta^2 + 1) c_1 > c_1$ for any $c_1 > 0$. We choose $c_1$ and $c_3$ such that $c_1 > \beta c_3$. Then, $(\beta^2 + 1) c_1 > c_1 > \beta c_3$, which implies that $(\beta^2 + 1) c_1 c_3 > \beta c_3^2$. The above and $c_1 > \beta c_3$ together imply that $\beta c_1^2 + c_1 c_3 > \beta c_3^2$.
So, with $\beta > 0$, the conditions $c_1 c_2 > c_3^2, c_1 > \beta c_3, \beta c_1 - \beta c_2 + c_3 = 0$ are feasible. We have two cases for
\begin{align*}
    \mu_1 = \min\{ \frac{\mu}{2m} \alpha c_1, (c_1- \beta c_3) \underline{\lambda}_L - \frac{2m L_f^2 + \mu \alpha c_1 L_f}{\mu \alpha c_1}\}.
\end{align*}

\noindent
{\bf Case-I}: $\mu_1 = \frac{\mu}{2m} \alpha c_1$. In this case,
\begin{align*}
    \alpha^2 L_f^2 c_3 < \beta \mu_1 \underline{\lambda}_L \iff \beta c_3 \leq \frac{\beta^2 c_1 \underline{\lambda}_L \mu}{2m L_f^2 \alpha}.
\end{align*}
There exists $\overline{\alpha}$ such that, for $\alpha < \overline{\alpha}$ we have $\frac{2m L_f^2 \alpha}{\beta^2 \underline{\lambda}_L \mu} < 1$, i.e., $\frac{\beta^2 c_1 \underline{\lambda}_L \mu}{2m L_f^2 \alpha} > c_1$. Thus, $c_1 > \beta c_3, \alpha^2 L_f^2 c_3 < \beta \mu_1 \underline{\lambda}_L$ are feasible for $\alpha < \overline{\alpha}$.

\noindent
{\bf Case-II}: $\mu_1 = (c_1- \beta c_3) \underline{\lambda}_L - \frac{2m L_f^2 + \mu \alpha c_1 L_f}{\mu \alpha c_1}$. We denote $q = \frac{2m L_f^2 + \mu \alpha c_1 L_f}{\mu \alpha c_1}$. Then, $\mu_1 = (c_1- \beta c_3) \underline{\lambda}_L - q$ and
\begin{align*}
    \alpha^2 L_f^2 c_3 < \beta \mu_1 \underline{\lambda}_L & \iff c_3 < \frac{\beta \underline{\lambda}_L((c_1- \beta c_3) \underline{\lambda}_L - q)}{\alpha^2 L_f^2} \\
    & \iff \beta c_3 < \frac{\beta^2 \underline{\lambda}_L (c_1 \underline{\lambda}_L - q)}{\alpha^2 L_f^2 + \beta^2 \underline{\lambda}_L^2}.
\end{align*}
We denote $p = \frac{1 + \beta^2 \underline{\lambda}_L^2}{\beta^2 \underline{\lambda}_L^2}$. Note that $p > 1$. Now,
\begin{align*}
    p = \frac{1 + \beta^2 \underline{\lambda}_L^2}{\beta^2 \underline{\lambda}_L^2} & \implies c_1 (1 - p \frac{\beta^2 \underline{\lambda}_L^2}{1 + \beta^2 \underline{\lambda}_L^2}) = 0 \\
    & \implies c_1 (1 - p \frac{\beta^2 \underline{\lambda}_L^2}{1 + \beta^2 \underline{\lambda}_L^2}) > - p \frac{\beta^2 \underline{\lambda}_L q}{1 + \beta^2 \underline{\lambda}_L^2} \\
    & \iff c_1 > p \frac{\beta^2 \underline{\lambda}_L (c_1 \underline{\lambda}_L - q)}{1 + \beta^2 \underline{\lambda}_L^2}.
\end{align*}
Then, $c_1 > p \frac{\beta^2 \underline{\lambda}_L (c_1 \underline{\lambda}_L - q)}{\alpha^2 L_f^2+ \beta^2 \underline{\lambda}_L^2} > \frac{\beta^2 \underline{\lambda}_L (c_1 \underline{\lambda}_L - q)}{\alpha^2 L_f^2+ \beta^2 \underline{\lambda}_L^2}$ for $\alpha < \frac{1}{L_f}$. Thus, $c_1 > \beta c_3, \alpha^2 L^2 c_3 < \beta \mu_1 \underline{\lambda}_L$ are feasible for $\alpha < \frac{1}{L_f}$.

\noindent 
From the above arguments, we conclude that~\eqref{eqn:ineq} is feasible under the conditions of Theorem~\ref{thm:thm2}.


Next, we use a first-order explicit Euler discretization of~\eqref{eqn:ctpi_x}-\eqref{eqn:ctpi_v} with fixed stepsize $h > 0$ to obtain the following algorithm from~\eqref{eqn:ctpi_x}-\eqref{eqn:ctpi_v}. For each $k \geq 0$ and each agent $i$,
\begin{align}
    & \bx_i(k+1) = \bx_i(k) + h K_i \sum_{j \in \N_i} (\bx_j(k) - \bx_i(k)) \nonumber \\
    & - h \beta K_i \sum_{j \in \N_i}(\bv_j(k) - \bv_i(k)) - h \alpha K_i \nabla f_i(\bx_i(k)), \label{eqn:dtpi_x} \\
    & \bv_i(k+1) = \bv_i(k) + h \beta K_i \sum_{j \in \N_i} (\bx_j(k) - \bx_i(k)). \label{eqn:dtpi_v}
\end{align}
The above algorithm can be compactly written as
\begin{align}
    & \begin{bmatrix} \bx(k+1) \\ \bv(k+1) \end{bmatrix} = \begin{bmatrix} \bx(k) \\ \bv(k) \end{bmatrix} \nonumber \\
    & - h \begin{bmatrix}
        K \L \bx(k) - \beta K \L \bv(k) + \alpha K \nabla F(\bx(k)) \\ \beta K \L \bx(k)
    \end{bmatrix}. \label{eqn:dtpi_comb}
\end{align}

The following result makes use of Lyapunov stability theory to prove linear convergence of the locally pre-conditioned algorithm~\eqref{eqn:dtpi_comb} under the same Assumptions of Theorem~\ref{thm:thm2}, leading to rate-matching or consistent discretization of~\eqref{eqn:ctpi_comb}. 

\begin{theorem} \label{thm:thm3}
Consider algorithm~\eqref{eqn:dtpi_comb} with initial condition $\bx(0), \bv(0) \in \R^{md}$. Suppose that Assumptions~\ref{assump_1}-\ref{assump_4} hold true. Then there exists $\overline{\alpha}, \overline{\beta}, \underline{\beta}, \overline{h} \in (0,\infty)$ such that, for $\alpha < \overline{\alpha}$, $\beta \in (\overline{\beta},\underline{\beta})$, and $h < \overline{h}$, the local estimate $\bx_i(k)$ linearly converges to the same $\bxs$ for each agent $\m$.
\end{theorem}

\begin{proof}
We define the error at iteration $k \geq 0$ as $\bZ(k) = \begin{bmatrix} \bz(k) \\ \by(k) \end{bmatrix} = \begin{bmatrix} \bx(k) - \bXs \\ \bv(k) - \mathbf{v_*} \end{bmatrix}$,
For $\bZ = [\bz^{\top}, \by^{\top}]^{\top} \in \R^{2md}$, let
\begin{align*}
\Tilde{F}(\bZ) = \begin{bmatrix}
        K \L \bz - \beta K \L \by + \alpha K (\nabla F(\bx) - \nabla F(\bXs)) \\ \beta K \L \bz
    \end{bmatrix}.
\end{align*}
Then, from~\eqref{eqn:dtpi_comb}, the error dynamics is
\begin{align}
    \bZ(k+1) = \bZ(k) - h \Tilde{F}(\bZ(k)). \label{eqn:dtpi_err}
\end{align}

Consider the case $\L \by(k) \neq 0_{md}$. As in the proof of Theorem~\ref{thm:thm2}, we define the Lyapunov function candidate $V: \R^{2md} \to \R$ such that $V(\bZ) = \frac{c_1}{2} \bz^{\top} K^{-1} \bz + \frac{c_2}{2} \by^{\top} K^{-1} \by - c_3 \bz^{\top} K^{-1} \by$ for $\bZ = [\bz^{\top} \by^{\top}]^{\top} \in \R^{2md}$, where $c_1, c_2, c_3 > 0$. Clearly, $\nabla V(\bZ)$ is Lipschitz continuous with some Lipschitz constant $\eta > 0$. Then, from Lemma~5 in~\cite{yi2020exponential}, we get
$V(\bZ(k+1) - V(\bZ(k)) \leq (\bZ(k+1) - \bZ(k))^{\top} \nabla V(\bZ(k)) + \frac{\eta}{2} \norm{\bZ(k+1) - \bZ(k)}^2$.
Along the trajectories of~\eqref{eqn:dtpi_err}, 
\begin{align}
    & V(\bZ(k+1) - V(\bZ(k)) \nonumber \\
    & \leq - h \Tilde{F}(\bZ(k))^{\top} \nabla V(\bZ(k)) + \frac{h^2 \eta}{2} \norm{\Tilde{F}(\bZ(k))}^2. \label{eqn:Vdiff}
\end{align}
Following the proof of Theorem~\ref{thm:thm2}, similar to~\eqref{eqn:Vdot_bd}-\eqref{eqn:V_bd_upp}:
\begin{align}
    - \Tilde{F}(\bZ(k))^{\top} \nabla V(\bZ(k)) & \leq - \epsilon_1 V(\bZ(k)), \label{eqn:Vdiff_1} \\
    V(\bZ(k)) & \geq \underline{\epsilon} \norm{\bZ(k)}^2. \label{eqn:Vz_low}
\end{align}
From the definition, $\|{\Tilde{F}(\bZ)}\|^2 \leq \norm{K}^2 (\beta^2 \overline{\lambda}_L^2 \norm{\bz}^2 + \norm{(\L \bz - \beta \L \by + \alpha(\nabla F(\bx) - \nabla F(\bXs)))}^2$.
Upon expanding the second term on the RHS above and utilizing AM-GM inequality, the definition of induced $2$-norm of matrices, and Lipschitz continuity of $\nabla F$ under Assumption~\ref{assump_4}, we obtain that there exists $\epsilon_2 > 0$ such that $\norm{\Tilde{F}(\bZ)}^2 \leq \epsilon_2 \norm{\bZ}^2$. Upon substituting above from~\eqref{eqn:Vz_low}, there exists $\epsilon_3 = \frac{\epsilon_2}{\underline{\epsilon}} > 0$ such that $\norm{\Tilde{F}(\bZ)}^2 \leq \epsilon_3 V(\bZ)$. Upon substituting from above and~\eqref{eqn:Vdiff_1} in~\eqref{eqn:Vdiff},
$V(\bZ(k+1)) \leq  (1 - h\frac{2\epsilon_1 - h \eta \epsilon_3}{2}) V(\bZ(k))$.
If $h \in (0, \frac{2\epsilon_1}{\eta \epsilon_3})$, then $\rho: = (1 - h\frac{2\epsilon_1 - h \eta \epsilon_3}{2}) < 1$. 

For $\L \by(k) = 0_{md}$, we consider the Lyapunov candidate $V_2: \R^{md} \to \R$ such that $V_2(\bz) = \frac{1}{2} \bz^{\top} K^{-1} \bz$ for $\bz \in \R^{md}$. Proceeding similarly as $V$ above, we obtain $V(\bZ(k+1)) \leq \rho_2 V_2(\bZ(k))$ for some $\rho_2 < 1$. The proof follows from the argument in the last paragraph of Theorem~\ref{thm:thm2}'s proof.
\end{proof}

\begin{remark}
The aforementioned analysis has assumed a unique solution $\bxs$ of~\eqref{eqn:opt_1}. When the solution set is not singleton, we denote the solution set of~\eqref{eqn:opt_1} by $\bx_{sol}$, and define $\bX_{sol} = [(\bx_{sol})^{\top}, \ldots, (\bx_{sol})^{\top}]^{\top}$. In this case, $F$ is said to satisfy the RSI condition if $(\nabla F(\bx) - \nabla F(\mathcal{P}_{\bX_{sol}}(\bx))^{\top} (\bx-\mathcal{P}_{\bX_{sol}}(\bx)) \geq \mu \norm{\bx-\mathcal{P}_{\bX_{sol}}(\bx)}^2 \, \forall \bx \in \R^{md}$, where $\mathcal{P}_{\bX_{sol}}(\bx) := \arg \min_{\by \in \bX_{sol}} \norm{\bx - \by}^2$ is the projection of $\bx$ onto $\bX_{sol}$~\cite{yi2020exponential}. We define $\bx^0 = [(\mathcal{P}_{\bX_{sol}}(\bx))^{\top}, \ldots, (\mathcal{P}_{\bX_{sol}}(\bx))^{\top}]^{\top}$. We assume that the gradients of the cumulative cost $F$ at any point on the solution set, i.e., $\{\nabla F(\bx) | \bx \in \bX_{sol} \}$ is a singleton. Under this assumption, the proofs of Theorem~\ref{thm:thm2}-\ref{thm:thm3} are valid upon replacing $\bXs$ with $\bx^0$. This is possible by invoking the properties of the projection operator from Theorem~1.5.5 in~\cite{nesterov2018lectures}.
\end{remark}

\section{Numerical Results}
\label{sec:exp}

\begin{figure*}[htb!]
\centering
\begin{adjustbox}{minipage=\linewidth,scale=0.8}
\begin{subfigure}{.33\textwidth}
  \begin{center}
  \includegraphics[width = \textwidth]{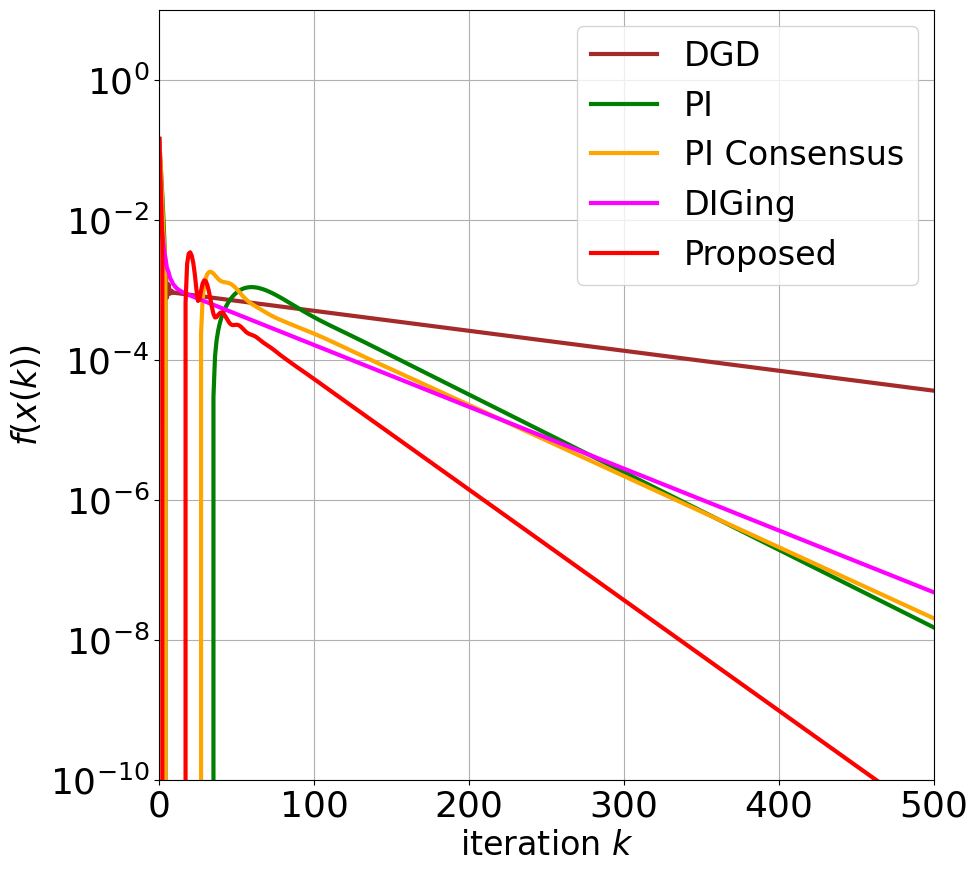}
  \caption{aggregate cost}
  \end{center}
\end{subfigure}%
\hfill
\begin{subfigure}{.33\textwidth}
  \begin{center}
  \includegraphics[width = \textwidth]{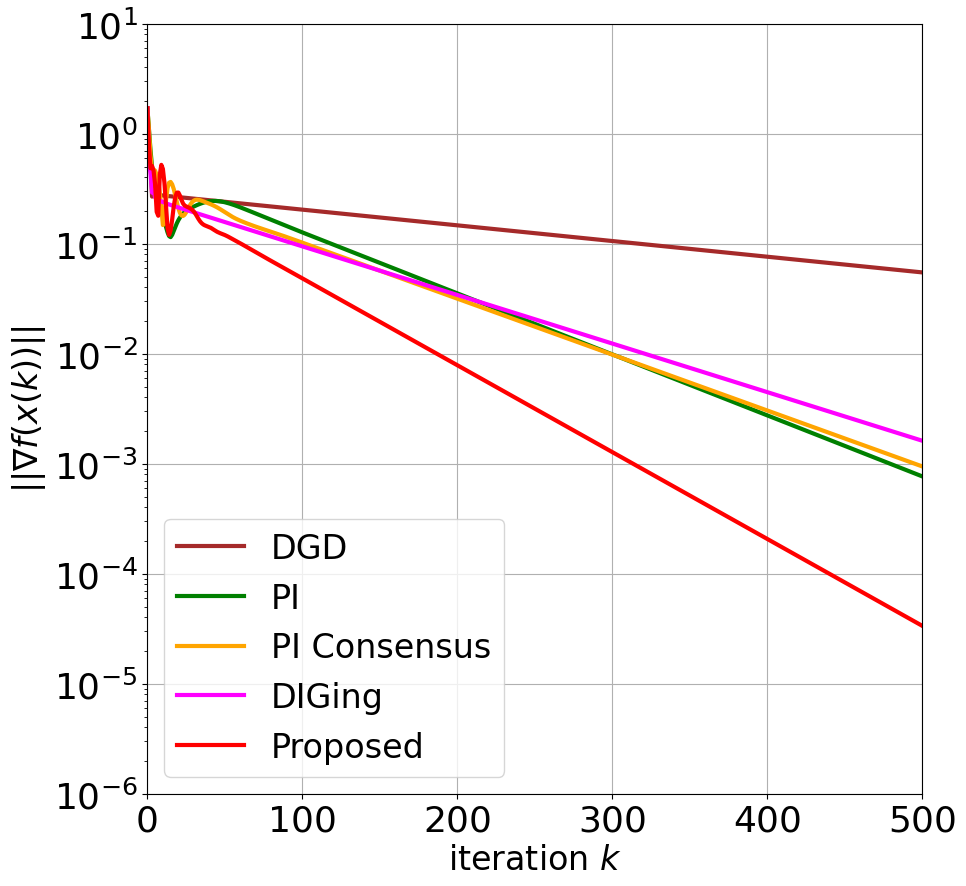}
  \caption{norm of gradient of aggregate cost}
  \end{center}
\end{subfigure}%
\hfill
\begin{subfigure}{.33\textwidth}
  \begin{center}
  \includegraphics[width = \textwidth]{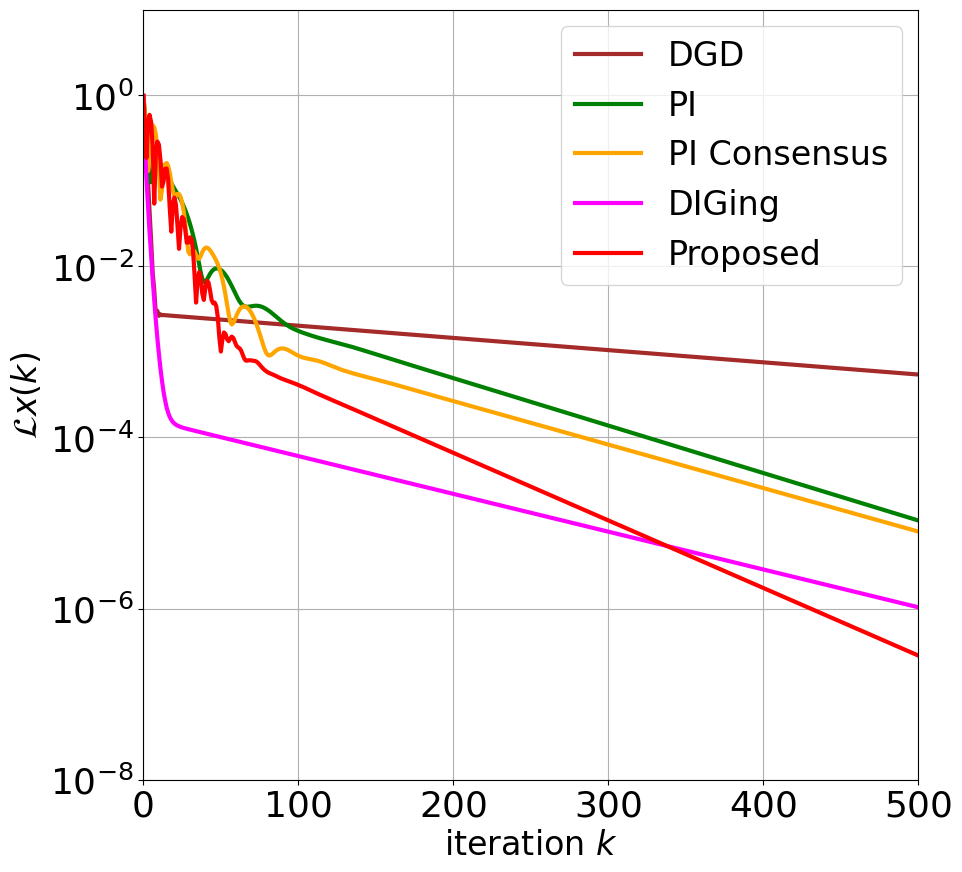}
  \caption{consensus term}
  \end{center}
\end{subfigure}%
\end{adjustbox}
\caption{\it Comparison of different distributed optimization algorithms for solving the problem from~\cite[Section VI]{yi2020exponential}.}
\label{fig:sc}
\vspace{-1em}
\end{figure*}
\vspace{-0.2em}

A suitable choice of $K_i$ could impact the convergence rate of PI consensus. For example, when $f_i$'s are convex, consider $K_i = \left(\nabla^2 f_i(\bx_i(0)) + \gamma I \right)^{-1}$ where $\gamma > 0$. In the case of single-agent, $K'_i = \left(\nabla^2 f_i(\bx_i(t))\right)^{-1}$ results in Newton's method which is fast. However, it requires computing the expensive inverse Hessian at each $t$. Instead, we limit computing the inverse Hessian only once, before initiating the algorithm. Since the matrix inverse for $K_i$ is computed once by each agent, it does not increase the per-iteration computational complexity of PI consensus. The parameter $\gamma$ extends this pre-conditioner to non-strongly convex $f_i$'s. Hence, we consider this choice of $K^i$.

To show the efficiency of~\eqref{eqn:dtpi_x}-\eqref{eqn:dtpi_v} with the proposed choice of pre-conditioner, we conduct two numerical simulations. First, we consider the problem from~\cite[Section VI]{yi2020exponential} with $m=5$ agents. Each $f_i$ is non-convex and satisfies the RSI condition (see~\cite{yi2020exponential}). Next, we consider the binary classification problem between the digits one and five using the MNIST training dataset. We consider the logistic regression model and conduct experiments to minimize the cross-entropy error on the training data. The data points are distributed equally among $m=5$ agents. The cumulative cost function does not satisfy Assumption~\ref{assump_4} for this problem, and there are multiple solutions of~\eqref{eqn:opt_1}. In both problems, the communication topology is a ring. For a fair comparison, the parameters in each algorithm are tuned such that the respective algorithm converges in fewer iterations. $\bx_i(0)$ in all the algorithms are the same, and each entry is chosen from the Normal distribution with zero mean and $0.1$ standard deviation. $\bv_i(0)$ in the proposed algorithm and PI consensus algorithm are chosen similarly, and $\bv_i(0)$ for the PI algorithm is according to~\cite{yi2020exponential}. Since $F$ is not strongly convex, the APM-C, Mudag, DAccGD, and ACC-SONATA algorithms are not applicable in both the problems, as their implementation requires the positive value of the strong-convexity coefficient. 

From Figure~\ref{fig:sc} and Figure~\ref{fig:mnist}, the PI consensus algorithm with the proposed $K_i$ converges much faster than the other algorithms. In the first example, the condition number of the Hessian of $F$ is of order $10$. In contrast, in the second example, the ratio between the largest and the smallest non-zero singular value of the Hessian of $F$ is of order $10^{11}$. Thus, even for ill-conditioned problems, the proposed algorithm can converge significantly faster than the other algorithms. Further, from Figure~\ref{fig:sc}, the PI consensus algorithm converges linearly under restricted secant inequality even though none of the local cost functions is convex, as proved in Theorem~\ref{thm:thm3}.

In~\eqref{eqn:dtpi_comb}, the effect of {\em connectivity} on the algorithm is represented by $\underline{\lambda}_{\min} (h \beta K \L)$. In the first problem, this value is $0.98 \underline{\lambda}_L$. Whereas, for PI consensus without preconditioning~\cite{wang2010control}, this value is $0.1 \underline{\lambda}_L$. Thus, introducing $K$ increases the ``effective connectivity'' of $\G$, which helps in faster consensus of the proposed algorithm (Fig.~\ref{fig:sc}(c)).

\begin{figure*}[htb!]
\centering
\begin{adjustbox}{minipage=\linewidth,scale=0.8}
\begin{subfigure}{.33\textwidth}
  \begin{center}
  \includegraphics[width = \textwidth]{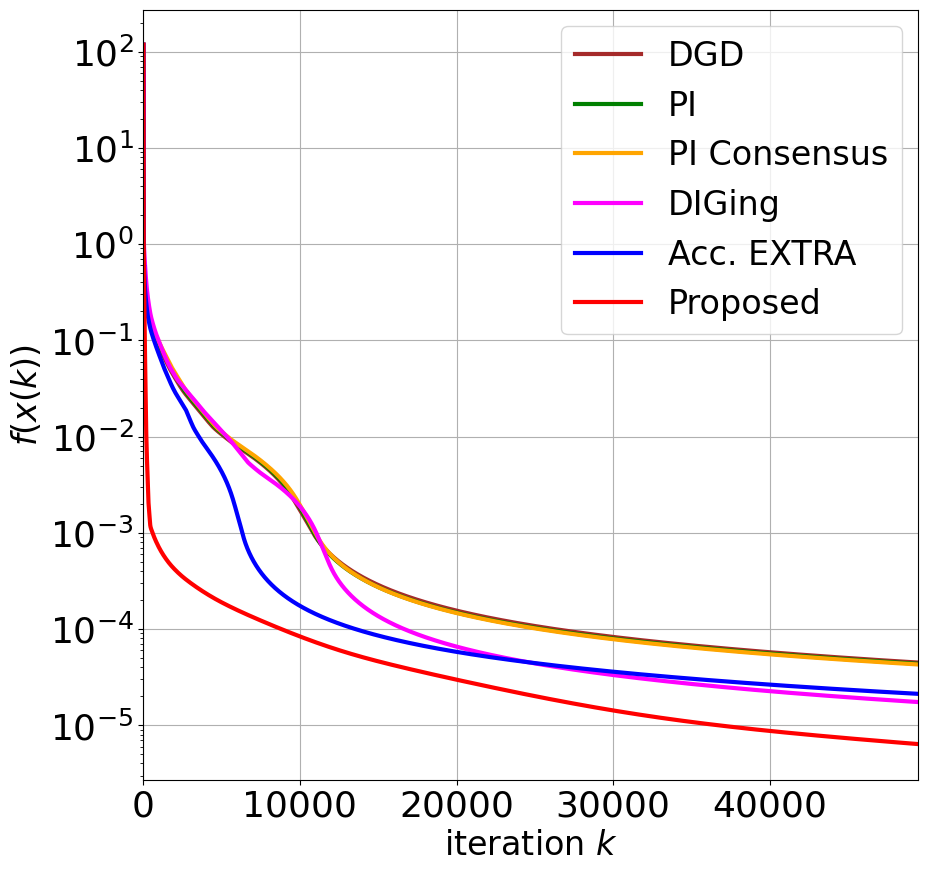}
  \caption{aggregate cost}
  \end{center}
\end{subfigure}%
\hfill
\begin{subfigure}{.33\textwidth}
  \begin{center}
  \includegraphics[width = \textwidth]{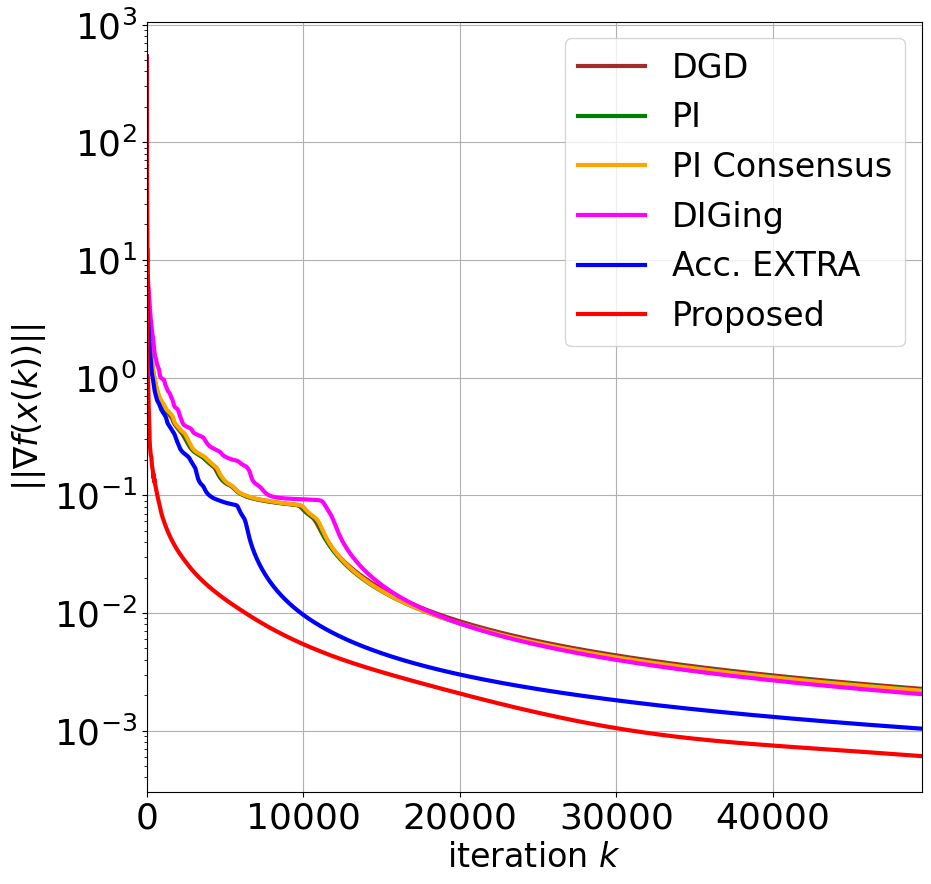}
  \caption{norm of gradient of aggregate cost}
  \end{center}
\end{subfigure}%
\hfill
\begin{subfigure}{.33\textwidth}
  \begin{center}
  \includegraphics[width = \textwidth]{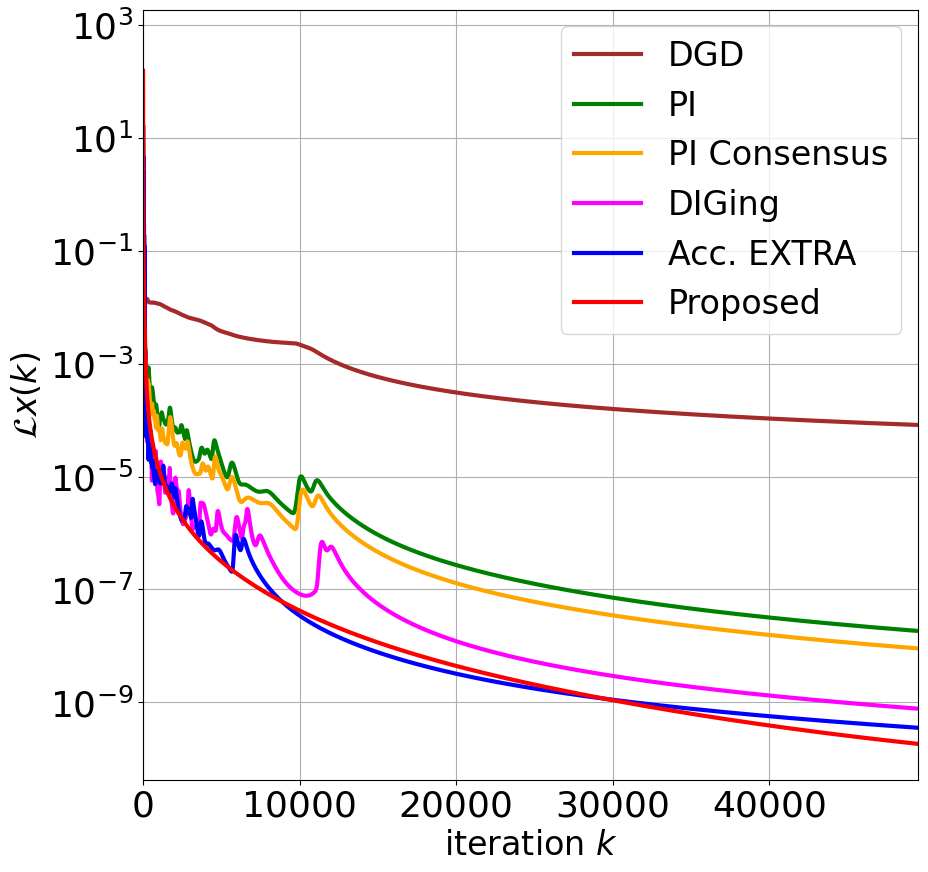}
  \caption{consensus term}
  \end{center}
\end{subfigure}%
\end{adjustbox}
\caption{\it Comparison of different distributed optimization algorithms for solving binary logistic regression on MNIST dataset.}
\label{fig:mnist}
\vspace{-1em}
\end{figure*}
\vspace{-0.2em}
\section{Conclusion}
For the first time, we presented convergence guarantee at an exponential rate of the PI consensus algorithm for cumulative cost functions satisfying smoothness and the RSI condition, without requiring convexity. Using Lyapunov stability theory, we proved the convergence of this algorithm in the continuous-time domain and obtained its rate-matching discretization. Further, we introduced a local pre-conditioning technique, locally computed by each agent, to accelerate the PI consensus algorithm. We demonstrated the efficacy of the proposed pre-conditioning compared to the existing distributed optimization algorithms, especially for ill-conditioned problems. 


\vspace{-0.2em}
\bibliographystyle{unsrt}        
\bibliography{refs} 

\addtolength{\textheight}{-12cm}

\end{document}